\documentclass[a4paper]{amsart}
\usepackage[utf8]{inputenc}

\usepackage{amssymb}
\usepackage{amsmath}
\usepackage{amsthm}
\usepackage{array}
\usepackage{mathtools}
\usepackage{multirow}
\usepackage{thmtools}
\usepackage{nicefrac}

\usepackage{lipsum}

\usepackage[colorlinks=true, citecolor=blue]{hyperref}
\usepackage{url}

\usepackage{todonotes}

\theoremstyle{plain}

\newtheorem{theorem}{Theorem}[section]
\newtheorem*{theorem*}{Theorem}
\newtheorem{corollary}[theorem]{Corollary}
\newtheorem{lemma}[theorem]{Lemma}
\newtheorem{proposition}[theorem]{Proposition}

\theoremstyle{definition}

\theoremstyle{remark}

\newcommand{\HH}{\mathbb{H}}
\newcommand{\RR}{\mathbb{R}}

\title[On ideal vertices of right--angled hyperbolic polyhedra]{On ideal vertices of right--angled \\ hyperbolic polyhedra}
\author{Stepan Alexandrov}
\address{Department of Discrete Mathematics, Moscow Institute of Physics and Technology}
\email{aleksandrov.sa@phystech.edu}
\urladdr{cyanprism.github.io}

\begin{document}

\begin{abstract}
	In this note, we improve Nikulin's inequality in the case of right--angled hyperbolic polyhedra. The new inequality allows to give much shorter proofs of the known dimension bounds. We also improve Nonaka's lower bound on the number of ideal vertices for right--angled hyperbolic polyhedra.
\end{abstract}

\maketitle

\section{Introduction}

A convex polyhedron is called \emph{right--angled} if all its dihedral angles are equal to $\pi/2$. Unlike in spherical or Euclidean spaces, the combinatorics of right--angled polytopes in hyperbolic space is more complex. For example, one may start with a regular right--angled dodecahedron in $\HH^3$. Then one can glue two such dodecahedra together by identifying a pair of mutually isometric pentagonal faces and obtain a new right--angled polytope. This procedure can be performed inductively to obtain a garland of dodecahedra. In the Euclidean case, the only right--angled polytope is a parallelogram, and any such garland would be just a parallelogram again. 

However, it is not trivial to show that hyperbolic right--angled polyhedra exist in $\mathbb{H}^n$ for an arbitrary $n\geq 2$ (while it is obvious in the Euclidean case). In fact, there exist upper bounds on $n$, and thus there are no right--angled hyperbolic (or, more generally, Coxeter hyperbolic) polyhedra in higher dimensions. 

\subsection{Dimension bounds for right--angled hyperbolic polyhedra}

It is known that there are no compact polyhedra in $\HH^{\geqslant 30}$ with all dihedral angles being integer submultiples of $\pi$ (\cite{Vin84}). Finite volume polyhedra of this type do not exist in $\HH^{\geqslant 996}$ (see \cite{Kho86, Pro86}). This bound can be significantly improved for right--angled polyhedra as follows.

\begin{theorem}[\cite{PV05}] \label{theorem: compact dimensions}
	There are no compact right--angled polyhedra in $\HH^{\geqslant 5}$.
\end{theorem}

Our first main result is a new short proof of the following two theorems.

\begin{theorem}[\cite{Kol12}] \label{theorem: ideal dimensions}
	There are no ideal right--angled polyhedra in $\HH^{\geqslant 7}$.
\end{theorem}

\begin{theorem}[\cite{PV05, Duf10}] \label{theorem: finite volume dimensions}
	There are no finite volume right--angled polyhedra in $\HH^{\geqslant 13}$.
\end{theorem}

The bound stated in Theorem~\ref{theorem: compact dimensions} is exact. The exactness of the other bounds is unknown: examples of ideal right--angled polyhedra are only known up to dimension $4$ and examples of finite volume right--angled polyhedra are only known up to dimension $8$.

Nikulin's inequality (\cite[Theorem~3.2.1]{Nik81}) states that low--dimensional faces of a simple Euclidean polytope cannot have too many faces on average. Khovanskii proved that Nikulin's inequality holds for polytopes that are simple at edges (\cite[Theorem~10]{Kho86}). We prove that Nikulin's inequality not only holds for polytopes that are simple at edges but can also sometimes be improved. Section~\ref{section: ideal polyhedra} deals with the inequality for $7$--dimensional ideal finite volume right--angled hyperbolic polyhedra and Section~\ref{section: finite volume polyhedra} deals with the inequality for $13$--dimensional right--angled hyperbolic polyhedra. Our proofs are based on the fact that every ideal vertex is contained in many facets. We believe that a more general case can be proved.

\subsection{Number of ideal vertices of right--angled hyperbolic polyhedra}

Lower bounds on the number of ideal vertices and facets of right--angled polyhedra are essential to prove the absence of such polyhedra in higher dimensions. Let $\mathcal{P}^n$ denote the family of finite volume non--compact right--angled hyperbolic polyhedra, $a_k(P)$ and $v_\infty(P)$ denote the number of $k$--faces and the number of ideal vertices of a polyhedron $P$ respectively. In \cite{Non15} Nonaka obtained some estimates on the number of ideal vertices of a right--angled hyperbolic polyhedron.

\begin{theorem}[{\cite[Main Theorem~1.2]{Non15}}] \label{theorem: ideal vertices}
	Let $P^n \in \mathcal{P}^n$. Then $v_\infty(P^n) \geqslant v_\infty^n$, where $v_\infty^n$ is defined by the following table.
	\[\renewcommand{\arraystretch}{1.2}
	\newcolumntype{C}[1]{>{\centering\arraybackslash$}p{#1}<{$}}
	\begin{array}{ | c || *{2}{C{1.5em}|} *{3}{C{2.5em}|} C{4em}| *{2}{C{4.5em}|} }
		\hline
		n & 5 & 6 & 7 & 8 & 9 & 10 & 11 & 12 \\
		\hline
		v_\infty^n & - & 3 & 17 & 36 & 91 & 254 & 741 & 2200 \\
		\hline
	\end{array}
	\]
\end{theorem}

We managed to improve Nonaka's estimates as follows.

\begin{theorem}
	Let $P^n \in \mathcal{P}^n$. Then $v_\infty(P^n) \geqslant v_\infty^n$, where $v_\infty^n$ is defined by the following table.
	\[\renewcommand{\arraystretch}{1.2}
	\newcolumntype{C}[1]{>{\centering\arraybackslash$}p{#1}<{$}}
	\begin{array}{ | c || *{2}{C{1.5em}|} *{3}{C{2.5em}|} C{4em}| *{2}{C{4.5em}|} }
		\hline
		n & 5 & 6 & 7 & 8 & 9 & 10 & 11 & 12 \\
		\hline
		v_\infty^n & 2 & 6 & 23 & 135 & 1704 & 182\,044 & 1.67 \cdot 10^9 & 1.27 \cdot 10^{17} \\
		\hline
	\end{array}
	\]
\end{theorem}

The proof of the theorem and more precise bounds are given in Section~\ref{section: cusps and facets}.

\subsection*{Acknowledgements}

The author is grateful to Nikolay Bogachev and Alexander Kolpakov for helpful discussions and remarks.

\subsection*{Funding}

The work was supported by the Theoretical Physics and Mathematics Advancement Foundation ``BASIS''.

\section{Right--angled hyperbolic polytopes}

The \emph{Minkowski space} $\RR^{n, 1}$ is the real vector space 
\[
	\RR^{n + 1} = \{ (x_0, x_1, \dots, x_n) \mid x_i \in \RR, \ i = 0, \dots, n\}
\]
 equipped with the indefinite scalar product
\[
	\langle x, y \rangle = - x_0 y_0 + x_1 y_1 + \dots + x_n y_n.
\]
Consider the two--sheeted hyperboloid 
\[
	H = \{x \in \RR^{n, 1} \mid \langle x, x \rangle = -1\}.
\] 
The \emph{hyperbolic space} is its upper half--sheet 
\[
	\HH^n = \{x \in \RR^{n, 1} \mid \langle x, x \rangle = -1, \ x_0 > 0\}
\] 
with the induced metric. Even though the Minkowski scalar product is indefinite, it induces a Riemannian metric on the hyperbolic space $\HH^n$. This metric turns out to have constant sectional curvature $-1$.  

The central projection of $\HH^n$ onto the plane $x_0 = 1$ through the origin produces an open ball. Its boundary is called the \emph{ideal boundary} $\partial \HH^n$ of the hyperbolic space $\HH^n$. The points of $\partial \HH^n$ correspond to the isotropic vectors 
\[
	\{x \in \mathbb{R}^{n,1} \mid \langle x, x \rangle = 0, \ x_0 > 0\} \mathbin{/} \RR_{>0}.
\]
The union $\overline{\HH^n} = \HH^n \cup \partial \HH^n$ is called the \emph{compactification} of $\HH^n$.

Any vector $e \in \RR^{n, 1}$ with $\langle e, e \rangle = 1$ defines the associated \emph{hyperbolic hyperplane} 
\[
	H_e = \HH^n \cap \{x \mid \langle e, x \rangle = 0\}
\] 
and the respective closed \emph{hyperbolic half--space} 
\[
	H^-_e = \HH^n \cap \{x \mid \langle e, x \rangle \leqslant 0\}.
\] 
By $\overline{H_e}$ we denote the closure of $H_e$ in $\overline{\HH^n}$. If $\langle e_1, e_1 \rangle = \langle e_2, e_2 \rangle = 1$ and $\langle e_1, e_2 \rangle \leqslant 0$ then the following holds:
\begin{enumerate}
	\item if $\langle e_1, e_2 \rangle > -1$, then the hyperplanes $H_{e_1}$ and $H_{e_2}$ intersect and the angle $\phi = \angle(H^-_{e_1}, H^-_{e_2})$ can be found from the equation $\cos \phi = -\langle e_1, e_2 \rangle$;
	\item if $\langle e_1, e_2 \rangle = -1$, then the hyperplanes $H_{e_1}$ and $H_{e_2}$ do not intersect while their closures $\overline{H_{e_1}}$ and $\overline{H_{e_2}}$ share a unique point on the boundary $\partial \HH^n$;
	\item if $\langle e_1, e_2 \rangle < -1$, then the closures $\overline{H_{e_1}}$ and $\overline{H_{e_2}}$ do not intersect, and the distance $\rho$ between $H_{e_1}$ and $H_{e_2}$ measured along their unique common perpendicular can be found from the equation $\cosh \rho = -\langle e_1, e_2 \rangle$.
\end{enumerate}

A \emph{convex hyperbolic $n$--dimensional polyhedron} $P$ is the intersection of finitely many closed half--spaces of $\HH^n$. We also assume that the interior of $P$ is non--empty. In the Klein model of the hyperbolic space $\HH^n$ the closure $\overline P \subset \overline{\HH^n}$ of a convex hyperbolic polyhedron $P \subset \HH^n$ is the intersection of a convex Euclidean polytope with the unit ball centred at the origin (see \cite{Vin93}). So, we apply basically the usual Euclidean terms (e.g., faces and vertices) to hyperbolic polyhedra.

We say that a vertex $v$ of $\overline P \subset \overline{\HH^n}$ is a \emph{finite vertex} if $v \in \HH^n$ and an \emph{ideal vertex} if $v \in \partial \HH^n$. A hyperbolic polyhedron has a finite volume if and only if it coincides with the convex hull of its vertices. A finite volume hyperbolic polyhedron is compact if and only if all of its vertices are finite. If a finite volume hyperbolic polyhedron $P$ has \emph{only} ideal vertices, then $P$ is called \emph{ideal}.

A \emph{right--angled hyperbolic polyhedron} is a convex hyperbolic polyhedron with all dihedral angles equal to $\frac{\pi}{2}$. Consider an $n$--dimensional finite volume right--angled hyperbolic polyhedron $P$. It is known that the neighbourhood of every finite and ideal vertex of $P$ is equal to the cone over a simplex and a cube, respectively. Hence, the polyhedron $P$ is \emph{simple at edges}, i.e.\ every edge is contained in exactly $n - 1$ facets.

\section{Combinatorics of convex Euclidean polytopes}

Consider an $n$--dimensional convex Euclidean polytope $P$. We say that $P$ is \emph{simplicial} if every facet of $P$ is a simplex. We say that $P$ is \emph{simple} if every vertex of $P$ is the intersection of exactly $n$ facets. The polar dual of a simplicial polytope is a simple one, and vice versa. 

Let $f_k$ denote the number of $k$--dimensional faces of $P$ (here and below we assume that $f_{-1} = f_{n} = 1$).

\begin{theorem}[Dehn--Sommerville equations]
	For every $n$--dimensional simplicial convex Euclidean polytope the following equations hold:
	\[
		\sum_{j = k}^{n - 1} (-1)^{j} \binom{j + 1}{k + 1} f_j = (-1)^{n - 1} f_k,
	\]
	where $k = -1, 0, \dots, n - 2$.
\end{theorem}

\begin{corollary}
	For every $7$--dimensional simplicial convex Euclidean polytope the following equalities hold:
	\[\setlength\arraycolsep{1pt}
	\begin{array}{*{9}{r}}
		f_3 &=& 5 f_2 &-& 15 f_1 &+& 35 f_0 &-&  70, \\
		f_4 &=& 9 f_2 &-& 34 f_1 &+& 84 f_0 &-& 168, \\
		f_5 &=& 7 f_2 &-& 28 f_1 &+& 70 f_0 &-& 140, \\
		f_6 &=& 2 f_2 &-&  8 f_1 &+& 20 f_0 &-&  40.
	\end{array}
	\]
\end{corollary}

\begin{corollary}
	For every $13$--dimensional simplicial convex Euclidean polytope the following equalities hold:
	\[\setlength\arraycolsep{1pt}
	\begin{array}{l *{14}{r}}
		f_6 &=& 8 f_5 &-& 36 f_4 &+& 120 f_3 &-& 330 f_2 &+& 792 f_1 &-& 1716 f_0 &+& 3432, \\
		f_7 &=& 27 f_5 &-& 159 f_4 &+& 585 f_3 &-& 1683 f_2 &+& 4125 f_1 &-& 9009 f_0 &+& 18018, \\
		f_8 &=& 50 f_5 &-& 325 f_4 &+& 1252 f_3 &-& 3685 f_2 &+& 9130 f_1 &-& 20\,020 f_0 &+& 40\,040, \\
		f_9 &=& 55 f_5 &-& 374 f_4 &+& 1474 f_3 &-& 4389 f_2 &+& 10\,934 f_1 &-& 24\,024 f_0 &+& 48\,048, \\
		f_{10} &=& 36 f_5 &-& 250 f_4 &+& 996 f_3 &-& 2982 f_2 &+& 7448 f_1 &-& 16\,380 f_0 &+& 32\,760, \\
		f_{11} &=& 13 f_5 &-& 91 f_4 &+& 364 f_3 &-& 1092 f_2 &+& 2730 f_1 &-& 6006 f_0 &+& 12\,012, \\
		f_{12} &=& 2 f_5 &-& 14 f_4 &+& 56 f_3 &-& 168 f_2 &+& 420 f_1 &-& 924 f_0 &+& 1848. \\
	\end{array}
	\]
\end{corollary}

\section{Ideal hyperbolic right--angled polyhedra} \label{section: ideal polyhedra}

Consider an ideal $7$--dimensional right--angled hyperbolic polyhedron $P$. Let $a_k$ denote the number of its $k$--dimensional faces (we use $f_k$ only for simplicial polytopes to avoid confusion). The polyhedron $P$ is combinatorially equivalent to a Euclidean polytope. Let us cut off all the vertices of the Euclidean polytope and denote the resulting truncated polytope by $P'$. Let $a'_k$ denote the number of $k$--dimensional faces of $P'$. The following equalities hold:
\[\setlength\arraycolsep{1pt}
\begin{array}{*{2}{ll @{\qquad}} ll}
	a'_0 &= 64 a_0, & 
	a'_1 &= a_1 + 192 a_0, & 
	a'_2 &= a_2 + 240 a_0, \\
	a'_3 &= a_3 + 160 a_0, & 
	a'_4 &= a_4 + 60 a_0, & 
	a'_5 &= a_5 + 12 a_0, \\
	& & a'_6 &= a_6 + a_0.
\end{array}
\]

The polyhedron $P$ is simple at edges. Therefore, the polytope $P'$ is simple, and its dual is simplicial. This fact allows us to apply the Dehn--Sommerville equations:
\[\setlength\arraycolsep{1pt}
\begin{array}{*{9}{r}}
	a_3 + 160 a_0 &=& 5 (a_4 + 60 a_0) &-& 15 (a_5 + 12 a_0) &+& 35 (a_6 + a_0) &-&  70, \\
	a_2 + 240 a_0 &=& 9 (a_4 + 60 a_0) &-& 34 (a_5 + 12 a_0) &+& 84 (a_6 + a_0) &-& 168, \\
	a_1 + 192 a_0 &=& 7 (a_4 + 60 a_0) &-& 28 (a_5 + 12 a_0) &+& 70 (a_6 + a_0) &-& 140, \\
	64 a_0 &=& 2 (a_4 + 60 a_0) &-&  8 (a_5 + 12 a_0) &+& 20 (a_6 + a_0) &-&  40.
\end{array}
\]
After simplifications we obtain: 
\[\setlength\arraycolsep{1pt}
\begin{array}{l *{10}{r}}
	a_3 &=& 5 a_4 &-& 15 a_5 &+& 35 a_6 &-&  5 a_0 &-&  70, \\
	a_2 &=& 9 a_4 &-& 34 a_5 &+& 84 a_6 &-& 24 a_0 &-& 168, \\
	a_1 &=& 7 a_4 &-& 28 a_5 &+& 70 a_6 &-& 38 a_0 &-& 140, \\
	0   &=& 2 a_4 &-&  8 a_5 &+& 20 a_6 &-& 20 a_0 &-&  40.
\end{array}
\]
Then the average number of vertices in a $2$--dimensional face of $P$ equals
\begin{multline*}
	\varkappa = \frac{192 a_0}{a_2} = 
	\frac{\nicefrac{192}{20} (2 a_4 - 8 a_5 + 20 a_6 - 40)}{9 a_4 - 34 a_5 + 84 a_6 - 24 a_0 - 168} = \\
	= \frac{\nicefrac{192}{20} (2 a_4 - 8 a_5 + 20 a_6 - 40)}{9 a_4 - 34 a_5 + 84 a_6 - \nicefrac{24}{20} (2 a_4 - 8 a_5 + 20 a_6 - 40) - 168} = \\
	= \frac{32}{11} \left(1 - \frac{2 a_5 - 6 a_6 + 12}{a_2}\right).
\end{multline*}

We claim that $a_5 > 3 a_6$ and therefore $\varkappa < \frac{32}{11} < 3$. The latter readily implies the absence of $7$--dimensional hyperbolic ideal right--angled polyhedra since every $2$--dimensional face contains at least three vertices. 

Indeed, every face of an ideal right--angled hyperbolic polyhedron is itself an ideal right--angled hyperbolic polyhedron. Every $6$--dimensional finite volume right--angled hyperbolic polyhedron has at least $27$ facets \cite[Lemma~1 and Proposition~4]{Duf10}. On the other hand, every $5$--dimensional face of $P$ is contained in exactly two $6$--dimensional faces of $P$. Thus, $2 a_5 \geqslant 27 a_6$.

\section{Finite volume hyperbolic right--angled polyhedra} \label{section: finite volume polyhedra}

Let $P$ be a finite volume hyperbolic $13$--dimensional right--angled polyhedron. Denote the number of its $k$--dimensional faces by $a_k$, the number of its finite vertices by $v_0$, and the number of its ideal vertices by $v_\infty$. Obviously, $a_0 = v_0 + v_\infty$. The polyhedron $P$ is combinatorially equivalent to a Euclidean polytope. Let us cut off all the vertices of the Euclidean polytope that correspond to the ideal vertices of $P$ and denote the obtained polytope by $P'$. Let $a'_k$ denote the number of $k$--dimensional faces of $P'$. The following equalities hold:
\[\setlength\arraycolsep{1pt}
\begin{array}{*{2}{llrrl @{\qquad}} llrrl}
	a'_0 &=& 4096\,v_\infty &+& v_0, & 
	a'_1 &=& 24\,576\,v_\infty &+& a_1, & 
	a'_2 &=& 67\,584\,v_\infty &+& a_2, \\
	a'_3 &=& 112\,640\,v_\infty &+& a_3, & 
	a'_4 &=& 126\,720\,v_\infty &+& a_4, & 
	a'_5 &=& 101\,376\,v_\infty &+& a_5, \\
	a'_6 &=& 59\,136\,v_\infty &+& a_6, & 
	a'_7 &=& 25\,344\,v_\infty &+& a_7, & 
	a'_8 &=& 7920\,v_\infty &+& a_8, \\
	a'_9 &=& 1760\,v_\infty &+& a_9, & 
	a'_{10} &=& 264\,v_\infty &+& a_{10}, & 
	a'_{11} &=& 24\,v_\infty &+& a_{11}, \\
	&&&&&
	a'_{12} &=& v_\infty &+& a_{12}.
\end{array}
\]

The polyhedron $P$ is simple at edges. Therefore, the polytope $P'$ is simple, and its dual is simplicial. This fact allows us to apply the Dehn--Sommerville equations. After simplifications we obtain:
\[\setlength\arraycolsep{.8pt}
\begin{array}{l *{16}{r}}
	a_6 &=& -1716 a_{12} &+& 792 a_{11} &-& 330 a_{10} &+& 120 a_9 &-& 36 a_8 &+& 8 a_7 &-& 132 v_\infty &+& 3432, \\
	a_5 &=& -9009 a_{12} &+& 4125 a_{11} &-& 1683 a_{10} &+& 585 a_9 &-& 159 a_8 &+& 27 a_7 &-& 1089 v_\infty &+& 18\,018, \\
	a_4 &=& -20\,020 a_{12} &+& 9130 a_{11} &-& 3685 a_{10} &+& 1252 a_9 &-& 325 a_8 &+& 50 a_7 &-& 3740 v_\infty &+& 40\,040, \\
	a_3 &=& -24\,024 a_{12} &+& 10\,934 a_{11} &-& 4389 a_{10} &+& 1474 a_9 &-& 374 a_8 &+& 55 a_7 &-& 6864 v_\infty &+& 48\,048, \\
	a_2 &=& -16\,380 a_{12} &+& 7448 a_{11} &-& 2982 a_{10} &+& 996 a_9 &-& 250 a_8 &+& 36 a_7 &-& 7116 v_\infty &+& 32\,760, \\
	a_1 &=& -6006 a_{12} &+& 2730 a_{11} &-& 1092 a_{10} &+& 364 a_9 &-& 91 a_8 &+& 13 a_7 &-& 3958 v_\infty &+& 12\,012, \\
	v_0 &=& -924 a_{12} &+& 420 a_{11} &-& 168 a_{10} &+& 56 a_9 &-& 14 a_8 &+& 2 a_7 &-& 924 v_\infty &+& 1848. \\
\end{array}
\]

Let $\alpha_0 = \binom{13}{2}$ and $\alpha_\infty = 12 \cdot 2^{11}$ denote the number of $2$--faces containing a finite and ideal vertex, respectively. The average number of vertices in a $2$--dimensional face of a finite volume right--angled hyperbolic $13$--polyhedron is equal to 
\[
	\varkappa = \frac{\alpha_0 v_0 + \alpha_\infty v_\infty}{a_2}.
\]
According to Nikulin inequality, $\varkappa < \frac{13}{3}$. Let us consider the following difference:
\begin{multline*}
	\left(\alpha_0 v_0 + \frac{10309}{6144} \cdot \alpha_\infty v_\infty\right) - \frac{13}{3} \cdot a_2 = \\
	2184 - \frac{26}{3} a_8 + 52 a_9 - 182 a_{10} + \frac{1456}{3} a_{11} - 1092 a_{12}.
\end{multline*}
The coefficients $\frac{13}{3}$ and $\frac{10309}{6144}$ are chosen so in order to cancel $a_7$ and $v_\infty$ on the right--hand side. 

We claim that the difference is negative. Indeed, every face of a hyperbolic finite volume right--angled polyhedron is itself a hyperbolic finite volume right--angled polyhedron. Every $9$--dimensional hyperbolic finite volume right--angled polyhedron has at least $152$ facets and every $11$--dimensional hyperbolic finite volume right--angled polyhedron has at least $564$ facets (\cite[Lemma~1 and Proposition~4]{Duf10}). However, every $8$--face is contained in exactly five $9$--faces and every $10$--face is contained in exactly three $11$--faces. Thus, $152 a_9 \leqslant 5 a_8$ and $564 a_{11} \leqslant 3 a_{10}$. Therefore, $52 a_9 - \frac{26}{3} a_8 < 0$ and $\frac{1456}{3} a_{11} - 182 a_{10} < 0$. Finally, $a_{12} > 2$, so $2184 - 1092 a_{12} < 0$.

Thus we proved that 
\[
	\frac{\alpha_0 v_0 + \frac{10309}{6144} \cdot \alpha_\infty v_\infty}{a_2} < \frac{13}{3}.
\]
The left part of the inequality is a weighted average number of the vertices in $2$--dimensional faces: the contribution of every finite vertex equals $1$ and the contribution of every ideal vertex equals $\frac{10309}{6144} > 1.6778$. Meanwhile, $2 \cdot \frac{10309}{6144} > 3.3557$ and $3 \cdot \frac{10309}{6144} > 5.03369$. Every $2$--face of a finite volume right--angled polyhedron is either contains at least $3$ ideal vertices, or $2$ ideal and $1$ finite vertices, or $1$ ideal and $3$ finite vertices, or $5$ finite vertices. Therefore, the weighted average number of the vertices in $2$ dimensional face is greater than $4.34 > \frac{13}{3}$. This contradicts the bound we obtained.

\section{Number of ideal vertices} \label{section: cusps and facets}

Recall that $\mathcal{P}^n$ denotes the family of finite volume non--compact right--angled hyperbolic polyhedra, $a_k(P)$ and $v_\infty(P)$ denote the number of $k$--faces and the number of ideal vertices of a finite volume right--angled hyperbolic polyhedron $P$ respectively. For a polyhedron $P$ denote by $a_k^l(P)$ the average number of $l$--faces of a $k$--face. In other words, 
\[
	a_k^l(P) = \frac{1}{a_k(P)} \sum_{\dim F = k} a_l(F),
\]
where $F$ runs over all $k$--faces of $P$.

\begin{proposition}[{\cite{Nik81}}, {\cite[Theorem~10]{Kho86}}] \label{proposition: Nikulin inequality}
	Let $P$ be an $n$--polytope that is simple at edges. Then
	\[
		a_k^l(P) < \binom{n - l}{n - k} \frac{\binom{\lceil n / 2 \rceil}{l} + \binom{\lfloor n / 2 \rfloor}{l}}{\binom{\lceil n / 2 \rceil}{k} + \binom{\lfloor n / 2 \rfloor}{k}}.
	\]
\end{proposition}

In \cite{Non15} Nonaka studied the right--angled hyperbolic $3$--polyhedra with a single ideal vertex and obtained the following result.

\begin{proposition}[{\cite[Corollary~3.6]{Non15}}] \label{proposition: facets of 3-polyhedra}
	If $P^3$ is a finite volume right--angled hyperbolic $3$--polyhedron and $v_\infty(P^3) \leqslant 1$, then $a_2(P^3) \geqslant 12$.
\end{proposition}

\begin{corollary} \label{corollary: cusps of 5-polyhedra}
	If $P^5 \in \mathcal{P}^5$, then $v_\infty(P^5) \geqslant 2$.
\end{corollary}

\begin{proof}
	Suppose that $v_\infty(P^5) \leqslant 1$. Every $3$--face of $P^5$ is a right--angled hyperbolic $3$--polyhedron of finite volume with at most one ideal vertex. Therefore, according to Proposition~\ref{proposition: facets of 3-polyhedra}, every $3$--face of $P$ contains at least $12$ facets and $a_3^2(P^5) \geqslant 12$. Meanwhile, according to Proposition~\ref{proposition: Nikulin inequality}, $a_3^2(P^5) < 12$.
\end{proof}

Let $\nu(P) = a_{n - 1}(P) + v_\infty(P)$ and if $\mathcal{P}^n \ne \varnothing$, let
\[
	\nu_n = \min_{P^n \in \mathcal{P}^n} \nu(P^n).
\]
Dufour used this value to prove Theorem~\ref{theorem: finite volume dimensions}. He discovered the following relations.

\begin{proposition}[{\cite[Proposition~4]{Duf10}}] \label{proposition: nu_5}
	$\nu_5 \geqslant 26$.
\end{proposition}

\begin{proposition}[{\cite[Lemma~1]{Duf10}}] \label{proposition: nu and facets}
	Let $P^n \in \mathcal{P}^n$ with $n \geqslant 3$. Then 
	\[
		a_{n - 1}(P^n) \geqslant 1 + \nu_{n - 1}.
	\]
\end{proposition}

\begin{proposition}[{\cite[Lemma~2]{Duf10}}] \label{proposition: nu_n}
	Let $n \geqslant 3$. Then 
	\[
		\nu_n \geqslant 5 - 2n + 2 \nu_{n - 1}.
	\]
\end{proposition}

Using the double counting technique, one can bound the number of the ideal vertices of a finite volume right--angled hyperbolic polyhedra from below as follows.

\begin{lemma} \label{lemma: ideal vertices}
	Let $P^n \in \mathcal{P}^n$.
	\[
		v_\infty(P^n) \geqslant \frac{a_{n - 1}(P^n) \cdot v_\infty'(P^n)}{2\,(n - 1)},
	\]
	where $v_\infty'(P^n)$ is the minimal number of ideal vertices that a facet of $P^n$ contains.
\end{lemma}

\begin{proof}
	Since every ideal vertex of the polytope $P^n$ is contained in exactly $2\:(n - 1)$ facets, the following inequality holds:
	\[
		v_\infty(P^n) \cdot 2\,(n - 1) \ = \
		\sum_{\dim F = n - 1} v_\infty(F) \ \geqslant \ 
		a_{n - 1}(P^n) \cdot v_\infty'(P^n),
	\]
	where $F$ runs over all facets of $P^n$.
\end{proof}

Let $P^n \in \mathcal{P}^n$. Proposition~\ref{proposition: nu_5} and Proposition~\ref{proposition: nu_n} imply that $\nu_5 \geqslant 26$, $\nu_6 \geqslant 45$, and $\nu_7 \geqslant 81$. Proposition~\ref{proposition: nu and facets}, Corollary~\ref{corollary: cusps of 5-polyhedra}, and Lemma~\ref{lemma: ideal vertices} provide
\begin{align*}
	a_5(P^6) \geqslant 27, \qquad &
	v_\infty(P^6) \geqslant \left\lceil \frac{27 \cdot 2}{10} \right\rceil = 6, \\
	a_6(P^7) \geqslant 46, \qquad &
	v_\infty(P^7) \geqslant \left\lceil \frac{46 \cdot 4}{12} \right\rceil = 23, \\
	a_7(P^8) \geqslant 82, \qquad &
	v_\infty(P^8) \geqslant \left\lceil \frac{82 \cdot 23}{14} \right\rceil = 135.
\end{align*}
Now applying the definition of $\nu_n$, Proposition~\ref{proposition: nu and facets}, and Lemma~\ref{lemma: ideal vertices} we obtain
\[\setlength\arraycolsep{1pt}
\def\skip{.3em}
\begin{array}{ ll @{\quad} rll @{\quad} ll }
	\nu_8 & \geqslant 217, &
	a_8 & (P^9) & \geqslant 218, &
	v_\infty(P^9) & \geqslant 1704, \\[\skip]
	
	\nu_9 & \geqslant 1922, &
	a_9 & (P^{10}) & \geqslant 1923, &
	v_\infty(P^{10}) & \geqslant 182\,044, \\[\skip]
	
	\nu_{10} & \geqslant 183\,967, &
	a_{10} & (P^{11}) & \geqslant 183\,968, &
	v_\infty(P^{11}) & \geqslant 1\,674\,504\,428, \\[\skip]

	\nu_{11} & \geqslant 1\,674\,688\,396, &
	a_{11} & (P^{12})  & \multicolumn{3}{l}{\geqslant 1\,674\,688\,397,} \\
	& & & & & v_\infty(P^{12}) & \geqslant 127\,466\,960\,740\,760\,088.
\end{array}
\]

\bibliographystyle{alpha}
\bibliography{refs.bib}

\begin{thebibliography}{Kho86}

\bibitem[Duf10]{Duf10}
G.~Dufour.
\newblock Notes on right-angled coxeter polyhedra in hyperbolic spaces.
\newblock {\em Geom. Dedicata}, 147:277--282, 2010.

\bibitem[Kho86]{Kho86}
A.~G. Khovanskii.
\newblock Hyperplane sections of polyhedra, toroidal manifolds and discrete
  groups in {L}obachevskii space.
\newblock {\em Funct. Anal. Appl.}, 20:41--50, 1986.

\bibitem[Kol12]{Kol12}
A.~Kolpakov.
\newblock On the optimality of the ideal right angled \(24\)-cell.
\newblock {\em Algebr. Geom. Topol.}, 12(4):1941--1960, 2012.

\bibitem[Nik81]{Nik81}
V.~V. Nikulin.
\newblock On the classification of arithmetic groups generated by reflections
  in {L}obachevsky spaces.
\newblock {\em Izv. Akad. Nauk SSSR Ser. Mat.}, 45:113--142, 1981.

\bibitem[Non15]{Non15}
J.~Nonaka.
\newblock The number of cusps of right-angled polyhedra in hyperbolic spaces.
\newblock {\em Tokyo J. Math.}, 38(2):539--560, 2015.

\bibitem[Pro86]{Pro86}
M.~N. Prokhorov.
\newblock The absence of discrete reflection groups with noncompact fundamental
  polyhedron of finite volume in {L}obachevskii space of large dimension.
\newblock {\em Izv. Akad. Nauk SSSR Ser. Mat.}, 50(2):413--424, 1986.

\bibitem[PV05]{PV05}
L.~Potyagailo and E.~B. Vinberg.
\newblock On right-angled reflection groups in hyperbolic spaces.
\newblock {\em Comment. Math. Helvetici}, 80:63--73, 2005.

\bibitem[Vin84]{Vin84}
E.~B. Vinberg.
\newblock Absence of crystallographic groups of reflections in {L}obachevsky
  spaces of large dimension.
\newblock {\em Tr. Mosk. Mat. Obs.}, 47:68--102, 1984.

\bibitem[Vin93]{Vin93}
E.~B. Vinberg, editor.
\newblock {\em Geometry II. Spaces of constant curvature}, volume~29 of {\em
  Encyclopaedia of Mathematical Sciences}.
\newblock Springer Berlin Heidelberg, 1993.

\end{thebibliography}

\end{document}